\documentclass[10pt,a4paper]{article}
\usepackage[latin1]{inputenc}
\usepackage{anysize}
\usepackage{graphicx}
\marginsize{2.5cm}{2.5cm}{2cm}{1cm}
\usepackage{amssymb,amsmath,amscd,amsfonts,amsthm,mathrsfs}
\newcommand{\bb}[1]{\mathbb{#1}}
\newcommand{\cc}[1]{\mathcal{#1}}

\numberwithin{equation}{section}
\newtheorem{theorem}{Theorem}
\newtheorem{lemma}[theorem]{Lemma}
\newtheorem{question}[theorem]{Question}
\usepackage{url}
\begin{document}
\title{A note on the combinatorial derivation of non-small sets}
\author{Joshua Erde \thanks{DPMMS, University of Cambridge}}
\maketitle
\begin{abstract}
Given an infinite group $G$ and a subset $A$ of $G$ we let $\Delta(A) = \{g \in G \,:\, |gA \cap A| =\infty\}$ (this is sometimes called the \emph{combinatorial derivation} of $A$). A subset $A$ of $G$ is called: \emph{large} if there exists a finite subset $F$ of $G$ such that $FA=G$; \emph{$\Delta$-large} if $\Delta(A)$ is large and \emph{small} if for every large subset $L$ of $G$, $(G \setminus A) \cap L$ is large. In this note we show that every non-small set is $\Delta$-large, answering a question of Protasov.
\end{abstract}
\section{Introduction}
For a subset $A$ of an infinite group $G$ we denote by

$$\Delta(A) = \{g \in G \,:\, |gA \cap A| = \infty\}.$$

This is sometimes called the \emph{combinatorial derivation} of $A$. We note that $\Delta(A)$ is a subset of $AA^{-1}$, the difference set of $A$. It can sometimes be useful to consider $\Delta(A)$ as the elements that appear in $AA^{-1}$ `with infinite multiplicity'. In \cite{P2011} Protasov analysed a series of results on the subset combinatorics of groups (see the survey \cite{S2011}) with reference to the function $\Delta$. These results were mainly to do with varying notions of the combinatorial size of a subset of a group. A subset $A$ of $G$ is said to be \cite{LP2009}:
\begin{itemize}
\item \emph{large} if there exists a finite subset $F$ of $G$ such that $FA=G$;
\item \emph{$\Delta$-large} if $\Delta (A)$ is large;
\item \emph{small} if $(G \setminus A) \cap L$ is large for every large subset $L$ of $G$.
\end{itemize}
Protasov asked \cite{P2011}:
\begin{question}\label{q:small}
Is every non-small subset of an arbitrary infinite group $G$ $\Delta$-large?
\end{question}

In this note we answer this question in the positive.

\begin{theorem}\label{t:small}
Let $G$ be an infinite group and $A$ a subset of $G$. If $A$ is not small, then $A$ is $\Delta$-large.
\end{theorem}

Our proof will hold in a slightly more general setting. Let us consider an arbitrary family of subsets $\cc{I} \subset \bb{P}(G)$. We will think of this family $\cc{I}$ as being in some way a set of subsets of $G$ that are insignificant in terms of their size. There are a few natural conditions that we will impose on the $\cc{I}$ which we consider. Firstly $\cc{I}$ should be closed under taking subsets, and also finite unions, we will call such an $\cc{I}$ an \emph{ideal}. Secondly $\cc{I}$ should be \emph{translation invariant}, that is for any $I \in \cc{I}$ and $g \in G$ we have that $gI = \{gi \,:\,i \in I\} \in \cc{I}$. Finally we will insist that $G \not\in \cc{I}$, that is $\cc{I} \neq \bb{P}(G)$, we call such an ideal \emph{proper}. The smallest non-trivial example of such a family is the set of finite subsets of $G$.\\
\ \\
Following on from Banakh and Lyaskovska \cite{BL2010}, given a translation invariant ideal $\cc{I} \subset \bb{P}(G)$ of an infinite group we say that $A =_{\cc{I}} B$ if the symmetric difference $A \triangle B \in I$. We say a subset $A$ of $G$ is:
\begin{itemize}
\item \emph{$\cc{I}$-large} if there exists a finite subset $F$ of $G$ such that $FA =_{\cc{I}} G$;
\item \emph{$\cc{I}$-small} if $(G \setminus A) \cap L$ is $\cc{I}$-large for every $\cc{I}$-large subset $L$ of $G$.
\end{itemize}

Similarly we can define $\Delta_{\cc{I}}(A) = \{g \in G \,:\, gA \cap A \not\in \cc{I}\}$, and say that a subset $A$ of $G$ is $\Delta_{\cc{I}}$-large if $\Delta_{\cc{I}}(A)$ is $\cc{I}$-large. In the case where $\cc{I}$ is the trivial ideal $\{\emptyset\}$, or when $\cc{I}$ is the ideal of finite subsets of $G$, these definitions agree with the ones above. Another natural example of such an $\cc{I}$ to consider would be the set of subsets of size less than a given cardinality $\kappa < |G|$. A less obvious example is the set of small sets, it is a simple check that for any proper translation invariant ideal $\cc{I}$ the set of $\cc{I}$-small sets, $\cc{S}_{\cc{I}}$, is also a proper translation invariant ideal. Banakh and Lyaskovska \cite{BL2010} showed that for any such $\cc{I}$ we have that $\cc{S}_{\cc{S}_{\cc{I}}} = \cc{S}_{\cc{I}}$, that is every set that is $\cc{S}_{\cc{I}}$-small is also $\cc{I}$-small. We show that when $\cc{I}$ is a translation invariant ideal the natural extension of Theorem \ref{t:small} holds.

\begin{theorem}\label{t:Ismall}
Let $G$ be an infinite group, $\cc{I}$ a proper translation invariant ideal of $G$ and $A$ a subset of $G$. If $A$ is not $\cc{I}$-small, then $A$ is $\Delta_{\cc{I}}$-large.
\end{theorem}

Thereom \ref{t:small} clearly follows by taking $\cc{I}$ to be the ideal of finite subsets of $G$, or the trivial ideal. We also go on to remark on some previous results relating to $\Delta$ in this new framework.

\section{Proof of Theorem \ref{t:Ismall}}
We note first that there is a natural ideal where the result can be seen to hold almost immediately. If we order the set of proper translation invariant ideals by inclusion we see that any chain of ideals has an upper bound, the union of the ideals in the chain, and so by Zorn's lemma there is some maximal such ideal $\cc{I}^*$. Since $\cc{I}^*$ is maximal $\cc{S}_{\cc{I}^*} = \cc{I}^*$ and so every $\cc{I}^*$-small set is a member of $\cc{I}^*$. Furthermore every subset of $G$ is either a member of $\cc{I}^*$ or $\cc{I}^*$-large. Indeed given $A \not\in \cc{I}^*$ we can consider the closure of the set $\cc{I}^* \cup A$ under taking subsets, finite unions and translations. This will be a translation invariant ideal $\cc{J}$ containing $\cc{I}^*$, and so since $\cc{I}^*$ maximal we have that $G \in \cc{J}$. Therefore, since $\cc{I}^*$ is an ideal, there is some finite subset $F$ of $G$ and some $I \in \cc{I}^*$ such that $FA \cup I = G$, that is $A$ is $\cc{I}^*$-large. We note that if $A$ is $\cc{I}^*$-small then $A$ cannot also be $\cc{I}^*$-large so if $\cc{L}^*$ is the set of $\cc{I}^*$-large sets we have that we can partition $\bb{P}(G)=\cc{I}^* \cup \cc{L}^*$. In \cite{P2011} (See also \cite{E2012}) it is shown that every large set is also $\Delta$-large, a similar argument shows that, when $\cc{I}$ is a translation invariant ideal, every $\cc{I}$-large set is also $\Delta_{\cc{I}}$-large.

\begin{lemma}\label{l:large}
Let $G$ be an infinite group, $\cc{I}$ a proper translation invariant ideal of $G$, $A$ a subset of $G$ and $F$ a finite subset of $G$. If $FA=_{\cc{I}} G$ then $F\Delta_{\cc{I}}(A)=G$. In particular if $A$ is $\cc{I}$-large then $A$ is $\Delta_{\cc{I}}$-large.
\end{lemma}
\begin{proof}
Let $F=\{f_1,f_2, \ldots, f_k\}$. We claim that for every $g \in G$ there is some $i$ such that $gA \cap f_iA \not\in \cc{I}$. Indeed, if $gA \cap f_iA \in \cc{I}$ for all $i$, then $\bigcup_i (gA \cap f_iA) = gA \cap (\bigcup f_iA) \in \cc{I}$. However we have that $\bigcup f_iA =_{\cc{I}} G$. Therefore we have that $gA =_{\cc{I}} gA \cap (\bigcup f_iA)$ and so $gA \in \cc{I}$, which implies that also $A \in \cc{I}$. But now we see that, since $G =_{\cc{I}} \bigcup f_iA$, $G \in \cc{I}$, contradicting the assumption that $\cc{I}$ is proper.\\ 
\ \\
Therefore for every $g \in G$ there is some $i$ such that $gA \cap f_iA \not\in \cc{I}$, and so $f_i^{-1}gA \cap A \not\in \cc{I}$. Hence for every $g \in G$ there is some $i$ such that $f_i^{-1}g \in \Delta_{\cc{I}}(A)$ and so $G=F\Delta_{\cc{I}}(A)$.
\end{proof}

Therefore, since every set which is not $\cc{I}^*$-small is $\cc{I}^*$-large and hence $\Delta_{\cc{I}^*}$-large, Theorem \ref{t:Ismall} holds for $\cc{I}^*$. In order to prove Theorem \ref{t:Ismall} for general $\cc{I}$ we will require the following lemma.
\begin{lemma}\label{l:union}
Let $G$ be an infinite group and $\cc{I}$ a proper translation invariant ideal of $G$. Let $X$ be a subset of $G$ such that there exists a finite subset $F$ of $G$ such that $FX =_{\cc{I}} G$. Then given a decomposition of $X$ into two sets $X=A \cup B$, either $F\Delta_{\cc{I}}(A) = G$ or there exists $g \in G$ such that $(g^{-1}F \cup \{e\}) B =_{\cc{I}} X$.
\end{lemma}
\begin{proof}
As before let $F = \{f_1, \ldots , f_k\}$. If $F\Delta_{\cc{I}}(A) \neq G$, then there exists $g \in G$, $g \not\in F\Delta_{\cc{I}}(A)$, that is, $f_i^{-1} g  \not\in \Delta(A)$ for $i=1,\ldots,k$. So the set $I_1$ is in $\cc{I}$ where

\[
I_1 = \bigcup_i \{h \in X \,:\, h \in A \text{ and } f_i^{-1} g h  \in A\}.
\]

Also we claim that the set $I_2$ is in $\cc{I}$ where

\[
I_2 = \{h \in X \, : \, f_i^{-1}gh \not\in X \text{ for all } i\}.
\]

Since if $F^{-1} g h \cap X = \phi$ then we have that $gh \cap FX = \phi$. Then since $FX =_{\cc{I}} G$, that is $FX = (G \setminus J)$ for some $J \in \cc{I}$, we have that $h \in g^{-1}J$.\\
\ \\
Therefore for all $h \in X \setminus (I_1 \cup I_2)$ and for all $i$, no pair $h$, $f_i^{-1} g h$ are both in $A$, and at least one of the group elements $f_i^{-1} g h$ is in $X$. Hence either $h$ or $f_i^{-1} g h$ must be in $B$. Therefore we have that $(g^{-1}F \cup \{e\})B =_{\cc{I}} X$.
\end{proof}

Lemma \ref{l:union} essentially says that whenever we decompose an $\cc{I}$-large set $X$ into two parts $X=A\cup B$, if $A$ is not $\Delta_{\cc{I}}$-large, then $B$ is $\cc{I}$-large. From this we can deduce Theorem \ref{t:Ismall}.

\begin{proof}[Proof of Theorem \ref{t:Ismall}]
If $A$ is not $\cc{I}$-small then there exists an $\cc{I}$-large set $L$ such that $(G \setminus A) \cap L$ is not $\cc{I}$-large. Without loss of generality let us assume that $A \subset L$. Then $L = (L \setminus A) \cup A$, and there exists a finite subset $F$ of $G$ such that $FL =_{\cc{I}} G$. Therefore, by Lemma \ref{l:union}, either $F \Delta_{\cc{I}}(A) \neq G$, and so $A$ is $\Delta_{\cc{I}}$-large, or there exists some $g \in G$ such that $(g^{-1}F \cup \{e\}) (L \setminus A) =_{\cc{I}} L$. But then $F(g^{-1}F \cup \{e\}) (L \setminus A) =_{\cc{I}} G$. However by assumption $L \setminus A$ was not $\cc{I}$-large, and so $F \Delta_{\cc{I}}(A) = G$. Therefore $A$ is $\Delta_{\cc{I}}$-large.
\end{proof}

In \cite{BP2003} Banakh and Protasov showed:

\begin{theorem}[Banakh and Protasov]\label{t:decom}
Let $G$ be an infinite group. Given a decomposition $G=A_1 \cup \ldots \cup A_n$ then there exists an $i$ and a subset $F$ of $G$ such that $|F| \leq 2^{2^{n-1}-1}$ and $FA_i A_i^{-1} = G$.
\end{theorem}

It is an old unsolved problem whether $i$ and $|F|$ can be chosen such that $FA_i A_i^{-1} = G$ and $|F| \leq n$. Noting that $\Delta(A_i) \subset A_iA_i^{-1}$, Protasov asked whether a similar result could hold true for some $\Delta(A_i)$. One can in fact prove a similar result, with the same bound on $|F|$, for $\Delta_{\cc{I}}$ by using Lemma \ref{l:union} inductively (See \cite{E2012}). However Banakh, Ravsky and Slobodianiuk \cite{BRS2013} were able to prove a a stronger result, replacing the bound $2^{2^{n-1}-1}$ with some function $\phi(n)$ which, whilst growing quicker than any exponential function, is eventually bounded by $n!$.

\begin{theorem}[Banakh, Ravsky and Slobodianiuk]\label{t:decom2}
Let $G$ be an infinite group, $\cc{I}$ a translation invariant ideal. Given a decomposition $G=A_1 \cup \ldots \cup A_n$ then there exists an $i$ and a subset $F$ of $G$ such that $|F| \leq \phi(n) := \max_{1<x \leq n} \frac{x^{n+1 - x} - 1}{x-1}$ and $F\Delta(A)_{\cc{I}} = G$.
\end{theorem}

\bibliography{Small}
\bibliographystyle{plain}
\end{document}